\documentclass[12pt]{article}

\title{A Tight Upper Bound on Acquaintance Time of Graphs}

\author{
Omer Angel\thanks{
{\tt angel@math.ubc.ca},
Department of Mathematics, University of British Columbia.}
\and
Igor Shinkar\thanks{
    {\tt igor.shinkar@weizmann.ac.il}
    Department of Computer Science and Applied Mathematics,
    Weizmann Institute of Science, Rehovot, {\sc Israel}.}
}

\usepackage{amsmath,amssymb,amsfonts,amsthm}

\usepackage[colorlinks=true,linkcolor=black,citecolor=black,urlcolor=blue,pdfborder={0 0 0}]{hyperref}

\usepackage{graphicx}

\usepackage[font=sf, labelfont={sf,bf}, margin=1cm]{caption}

\usepackage{cleveref}
  \crefname{theorem}{Theorem}{Theorems}
  \crefname{thm}{Theorem}{Theorems}
  \crefname{lemma}{Lemma}{Lemmas}
  \crefname{lem}{Lemma}{Lemmas}
  \crefname{remark}{Remark}{Remarks}
  \crefname{prop}{Proposition}{Propositions}
  \crefname{defn}{Definition}{Definitions}
  \crefname{corollary}{Corollary}{Corollaries}
  \crefname{section}{Section}{Sections}
  \crefname{figure}{Figure}{Figures}

\newtheorem{thm}{Theorem}[section]

\newtheorem{corollary}[thm]{Corollary}

\numberwithin{equation}{section}

\newcommand{\N}{\mathbb N}
\newcommand{\AC}{\mathcal{AC}}
\newcommand{\floor}[1]{{\lfloor#1\rfloor}}
\newcommand{\ceil}[1]{{\lceil#1\rceil}}

\begin{document}

\maketitle
\begin{abstract}
  In this note we confirm a conjecture raised by Benjamini
  et~al.~\cite{BST} on the acquaintance time of graphs, proving that for
  all graphs $G$ with $n$ vertices it holds that $\AC(G) = O(n^{3/2})$,
  which is tight up to a multiplicative constant.  This is done by proving
  that for all graphs $G$ with $n$ vertices and maximal degree $\Delta$ it
  holds that $\AC(G) \leq 20 \Delta n$.  Combining this with the
  bound $\AC(G) \leq O(n^2/\Delta)$ from \cite{BST} gives the foregoing
  uniform upper bound of all $n$-vertex graphs.

  We also prove that for the $n$-vertex path $P_n$ it holds that
  $\AC(P_n)=n-2$. In addition we show that the barbell graph $B_n$ consisting
  of two cliques of sizes $\ceil{n/2}$ and $\floor{n/2}$ connected by a single
  edge also has $\AC(B_n) = n-2$. This shows that it is possible to add
  $\Omega(n^2)$ edges to $P_n$ without changing the $\AC$ value of the graph.
\end{abstract}

\section{Introduction}\label{sec:intro}

In this note we study the following graph process, recently introduced
by Benjamini et~al.\ in \cite{BST}. Let $G = (V,E)$ be a finite connected graph.
Initially we place one agent in each vertex of the graph.
Every pair of agents sharing a common edge are declared to be acquainted.
In each round we choose some matching of $G$ (not necessarily a maximal matching),
and for each edge in the matching the agents on this edge swap places,
which allows more agents to become acquainted.
A sequence of matchings that allows all agents to meet
is called {\em a strategy for acquaintance in $G$}.
The \emph{acquaintance time of $G$}, denoted by $\AC(G)$,
is the minimal number of rounds in a strategy for acquaintance in $G$.

It is trivial that for an $n$ vertex graph $G=(V,E)$ it holds that
$\AC(G) \leq O(n^2)$ since every agent can meet all others by
traversing the graph along some spanning tree in at most $2n$ rounds.
Benjamini~et~al.~\cite{BST} proved an asymptotically smaller upper bound
of $\AC(G) = O(n^2 \cdot \log\log(n)/ \log(n))$ for all graphs with $n$ vertices.
This bound has been then improved by Kinnersley~et~al.~\cite{KMP} to
$\AC(G) = O(n^2 / \log(n))$.
In this note we prove that $\AC(G) = O(n^{1.5})$
for all graphs $G$ with $n$ vertices, which is tight up
to a multiplicative constant. Indeed, by Theorem 5.1 in \cite{BST}
for every function $f:\N \to \N$ that satisfies $1 \leq f(n) \leq n^{1.5}$
there is a family of graphs $\{G_n\}_{n \in \N}$ such that $G_n$
has $n$ vertices and $\AC(G_n) = \Theta(f(n))$.

We also prove that for $P_n$, an $n$-vertex path, we have $\AC(P_n) = n-2$.
For the upper bound we show a $(n-2)$-rounds strategy for acquaintance in
$P_n$. For the lower bound we prove that the barbell graph $B_n$ consisting
of two cliques of sizes $\ceil{n/2}$ and $\floor{n/2}$ connected by a single
edge satisfies $\AC(B_n) = n - 2$. This shows that it is possible to add
$\Omega(n^2)$ edges to $P_n$ without changing the $\AC$ value of the graph.

\section{Upper Bound on Acquaintance Time of Graphs}\label{sec:main}

The following theorem is the main result of this paper.

\begin{thm}\label{thm:main}
  Let $G=(V,E)$ be a graph with $n$ vertices, and suppose that the maximal
  degree of $G$ is $\Delta$. Then $\AC(G) \leq 20 \Delta n$.
\end{thm}

\begin{proof}
  Clearly, removing edges from $G$ can only increase its acquaintance time.
  Thus, in order to upper bound $\AC(G)$ we may fix a spanning tree of $G$
  and use only the edges of the tree, and so, we henceforth assume that $G$
  is an $n$-vertex tree. A contour of the tree is a cycle
  that crosses each edge exactly twice, and visits each vertex $v$ a number
  of times equal to its degree.  Such a contour is obtained by considering
  a DFS walk on $G$ (see \cref{fig:contour}). We remove an edge from the
  contour to get a path $\Gamma$ in $G$ of length $2n-3$, that visits every
  vertex at most $\Delta$ times.

  \begin{figure}[htbp]
    \centering
    \includegraphics[height=80mm]{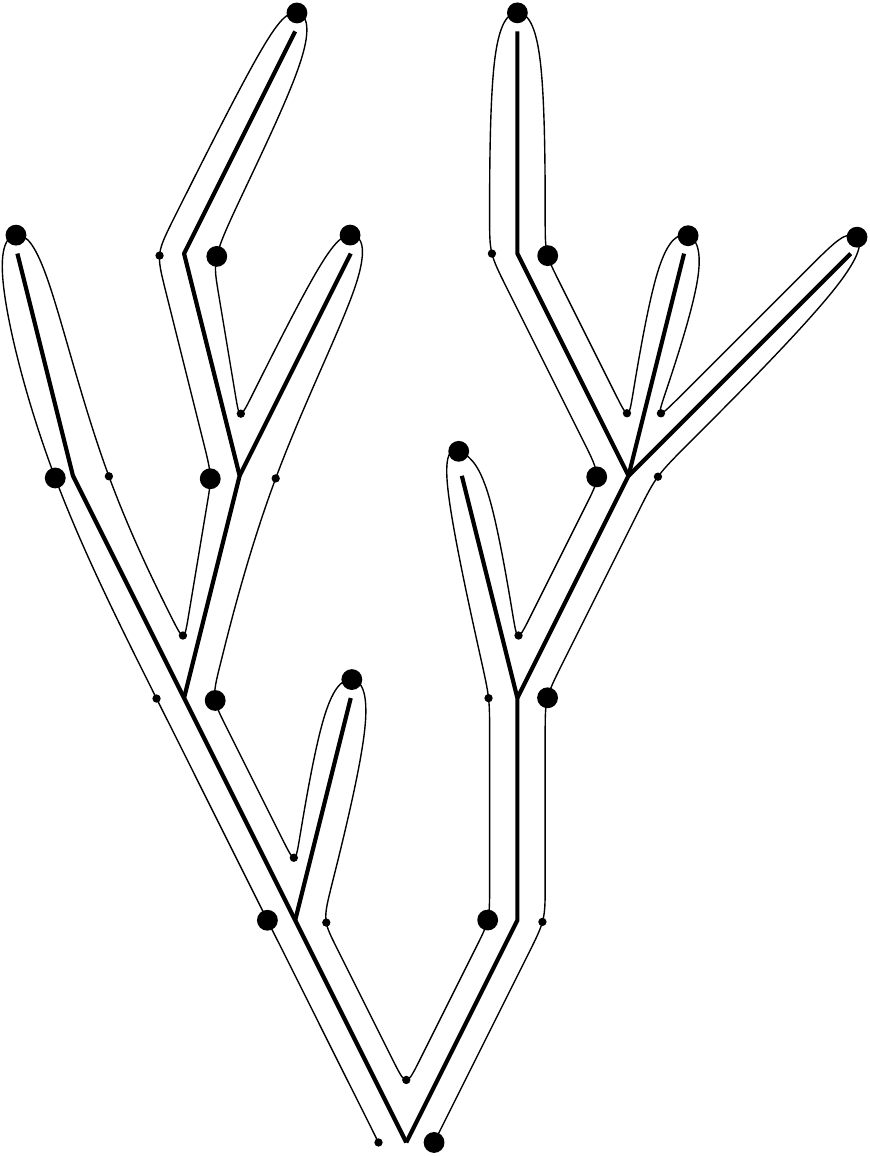}
    \caption{A tree with a marked contour.}
    \label{fig:contour}
  \end{figure}

  Let $\pi$ be the projection from $\Gamma$ to $G$.
  We first argue that it is possible to choose $n$ vertices on the path
  $\Gamma$ that project to distinct (and hence all) vertices of $G$, so
  that the gaps between the chosen consecutive vertices along $\Gamma$
  are at most $3$. To do this, we need to pick one vertex of $\Gamma$ from
  $\pi^{-1}(x)$ for each $x\in G$.  Fix a root for the tree at which
  $\Gamma$ starts.  In a vertex $x$ in an even level we pick the first
  vertex of $\Gamma$ projecting to $x$.  For $x$ in an odd level we pick
  the last one. See \cref{fig:contour} for an example.

  Note that $\Gamma$ only visits leaves of the tree once. Between leaves
  the contour descends some levels towards the root, and then ascends.
  Along the descent vertices are visited for the last time, and so every
  other vertex is selected.  Along the ascent vertices are visited for the
  first time.  Consequently, it is not possible to have more than three
  steps of $\Gamma$ between consecutive marked vertices.

  Consider the following $n$-rounds strategy. In
  even rounds we swap the edges $\{(i,i+1) : i \mbox{ even}\}$, and in odd rounds we swap
  the edges $\{(i,i+1) : i \mbox{ odd}\}$. It is easy to see that after $n$ rounds the agents
  are in reversed order on the path, and so every pair of agents must have swapped places.

  The $n$ agents on the vertices of $G$ can be seen as being on the vertices
  of $\Gamma$, where we use the marks specified above to decide
  which vertex of $\Gamma$ is occupied.  In order to present a $O(\Delta
  \cdot n)$-rounds strategy for acquaintance in $G$ we emulate the strategy
  for the path $\Gamma$, except that our goal is to make the $n$ agents
  located in the marked vertices of $\Gamma$ swap places, and hence meet.
  This is done by simulating each round of the strategy for $\Gamma$ by a
  sequence of at most $20\Delta$ matchings.

  In order to swap a consecutive pair of agents $p_i$ and $p_j$ in vertices
  $i$ and $j$ we can perform a sequence of swaps in $\Gamma$, namely $(i,i+1),
  \dots, (j-1,j)$, which brings the agent $p_i$ to the vertex $j$, followed
  by the sequence $(j-1,j-2), \dots, (i+1,i)$, bringing the agent $p_j$ to
  the vertex $i$.  This projects to swaps on $G$ that exchange the
  agents at $\pi(i)$ and $\pi(j)$ and leaves all others unchanged. The
  gaps between consecutive agents are at most $3$ so it takes at most $5$
  steps on $G$ to perform such a swap.

  The difficulty is that swapping between a pair of agents $p_i$ and $p_j$
  could interfere with swapping another pair $p_{i'}$ and $p_{j'}$, which can
  happens if the projections of the intervals $[i,j]$ and $[i',j']$ in the
  path $\Gamma$ intersect in $G$. If not for this problem, we would have
  a $5n$ round acquaintance strategy for $G$.

  In order to solve this problem, we shall separate each round into several
  sub-rounds, so that conflicting pairs are in different sub-rounds.  Since
  $\Gamma$ visits each vertex of $G$ at most $\Delta$ times, and since the
  intervals $[i,j]$ of $\Gamma$ that we care about are disjoint, each
  vertex of $G$ is contained in at most $\Delta$ such intervals.  Each
  interval consists of at most $4$ vertices of $G$, and therefore each pair
  $[i,j]$ is in conflict with less than $4\Delta$ other pairs $[i',j']$.
  We can assign each pair one of $4\Delta$ colors, so that conflicting
  pairs have different colors.  We now split the round into $20\Delta$
  sub-rounds where in $5$ consecutive sub-rounds we swap all pairs of
  color $i$ that are to be swapped in that round of the path strategy.

  Each round of the strategy on $P_n$ can be simulated by $20\Delta$ rounds
  on $G$, and hence $\AC(G) \leq 20 \Delta n$.
  This completes the proof of the theorem.
\end{proof}

As an immediate corollary from \cref{thm:main} we obtain the following
uniform upper bound on the acquaintance time of graph with $n$ vertices.

\begin{corollary}\label{cor:uniform bound}
  For all $n$-vertex graphs $G$ it holds that $\AC(G) = O(n^{3/2})$.
\end{corollary}

\begin{proof}
  We have that $\AC(G) \leq \min(O(n^2/\Delta),O(n\Delta)) \leq O(n^{3/2})$,
  where the two bounds are from \cref{thm:main} and Claim~5.7 of~\cite{BST}.
\end{proof}

Note that if $G$ is not a tree then we can try to improve our bound by
finding a spanning tree with smaller degrees.  For example, the giant
component of $G(n,p)$ with $p=c/n$ has maximal degree of order
$\frac{\log n}{\log\log n}$, but has a spanning tree with bounded degrees,
and so has acquaintance time of order $n$.

\section{Exact calculation of $\AC(P_n)$ and $\AC(B_n)$}
\label{sec:P_n}

In this section we compute $\AC(P_n)$ the acquaintance time of the
$n$-vertex path.

\begin{thm}\label{thm:P_n}
  Let $P_n$ be a path with $n$ vertices, and let $B_n$ be the barbell graph
  consisting of cliques of sizes $\ceil{n/2}$ and $\floor{n/2}$ connected
  by a single edge.  Then
  \[
  \AC(P_n) = \AC(B_n) = n-2.
  \]
\end{thm}

\begin{proof}
  We first prove that $\AC(P_n) \leq n-2$ by describing a $(n-2)$-rounds
  strategy for acquaintance in $P_n$.  Then we prove that
  $\AC(B_n) \geq n-2$.  This is clearly enough for the proof of
  the theorem as $P_n$ is contained in $B_n$.

  In order to prove that $\AC(P_n) \leq n-2$ consider the strategy that in
  odd-numbered rounds flips all edges $\{(i,i+1) : i \mbox{ odd}\}$, and in
  the even-numbered rounds swaps all edges $\{(i,i+1) : i \mbox{ even}\}$.
  Consider the walk performed by an agent that begins in some odd-indexed
  vertex under this strategy.  The agent will move one step up in each
  round until reaching the vertex $n$, will stay there for one round, and
  then move down one step in each round.  Similarly, an agent starting at
  an even vertex will move down until reaching the vertex $1$, stay there
  for one round and the move up.

  After $n$ rounds, the agent who started in position $i$ is in position
  $n+1-i$, and in particular every pair of agents have already met.
  We claim that in fact all agents are acquainted two rounds earlier.
  Indeed, consider two agents $p_i$ and $p_j$ who started in non-adjacent
    the vertices $i \leq j-2$ respectively. The proof follows by considering
    the following 3 cases.
    \begin{enumerate}
    \item{\bf $|i-j|$ is even:}
        Assume for concreteness that $i$ and $j$ are odd.
        (The case of $i$ and $j$ even is handled similarly)
        Then, $p_i$ meets $p_j$ in one of the first $n-i-1$ rounds since after
        the $(n-i-1)$'st rounds the agent $p_i$ reaches the vertex $n-1$.
    \item{\bf $i$ is odd and $j$ is even:}
        In this case the agents move towards each other, and hence meet in
        the $(j-i-2)$'nd round.
    \item{\bf $i$ is even and $j$ is odd:}
        Then, the agent $p_i$ reaches the vertex $1$ after $i-1$ rounds,
        stays there for another round, and then moves up.
        Therefore, in the $t$'th round the agent $p_i$ visits
        the vertex $t-i+1$ for all $i \leq t \leq n-2$.
        Analogously, for all $n-j < t \leq n-2$ the agent $p_j$ visits
        in $t$'th round the vertex $2n-(t+j-1)$.
        This implies that in round number $t = n - \frac{j-i+1}{2}$
        the agents $p_i$ and $p_j$ are located in neighboring vertices
        $n - \frac{i+j-1}{2}$ and $n - \frac{i+j-1}{2}+1$ respectively.
    \end{enumerate}
    This completes the proof of the first part of the proof, namely $\AC(P_n) \leq n-2$.

\medskip
    For the lower bound consider the barbell graph $B_n$ consisting of two
    disjoint cliques of sizes $\ceil{n/2}$ and $\floor{n/2}$ connected by a
    single edge, called the bridge. We claim that $\AC(B_n) \geq n-2$.

    Suppose there is an $m$-round strategy for acquaintance in $B_n$
    with $k$ swaps across the bridge.
    Any agent involved in such a swap is immediately
    acquainted with all others. Call these agents good.
    If the strategy has $k$ swaps, then $2k$ of the $m+1$
    configurations (those before and after the bridge-swaps)
    have good agents at both endpoints of the bridge.

    Note that a second consecutive swaps across the bridge achieves
    nothing, and also that there is also no point in swapping across
    edges not incident with the bridge.
    Hence, if there are $k$ swaps across the bridge, then the number of bad
    agents in the two cliques are at least $\ceil{n/2}-k$ and $\floor{n/2}-k$.
    These agents can only be acquainted by being by the bridge simultaneously,
    which requires at least $(\ceil{n/2}-k) \cdot (\floor{n/2}-k)$ configurations.
    Therefore, we get
    \[
        m+1 \geq 2k + (\ceil{n/2}-k)(\floor{n/2}-k) = k^2 -(n-2)k + \ceil{n/2}\floor{n/2}.
    \]
    This is minimized for $k = n/2-1$, giving a lower bound of
    $m+1 \geq n-1$ for even values of $n$, and $m+1 \geq n-5/4$ for odd $n$.
    This clearly suffices since $m$ is an integer.
\end{proof}




\begin{thebibliography}{}

\bibitem[BST13]{BST}
I.~Benjamini, I.~Shinkar, and G.~Tsur.
\newblock Acquaintance time of a graph.
\newblock 2013.
\newblock http://arxiv.org/abs/1302.2787.

\bibitem[KMP13]{KMP}
W.B. Kinnersley, D.~Mitsche, and P.~Pra{\l}at.
\newblock A note on the acquaintance time of random graphs.
\newblock 2013.
\newblock http://arxiv.org/abs/1305.1675.

\end{thebibliography}
\end{document}